\documentclass[10pt,a4paper]{amsart}
\usepackage{verbatim}

\usepackage[T1]{fontenc}
\usepackage{graphicx}
\usepackage{enumerate}
\usepackage{amsmath,amsfonts,amssymb}
\usepackage{color}
\usepackage{cite}
\definecolor{citation}{rgb}{0.2,0.58,0.2} 
\definecolor{formula}{rgb}{0.1,0.2,0.6}
\definecolor{url}{rgb}{0.3,0,0.5}
\usepackage{pgf,tikz}
\usepackage{mathrsfs}
\usepackage{fancyhdr}
\usepackage{dutchcal}

\usepackage{dsfont}

\newcommand{\reqnomode}{\tagsleft@false}

\usepackage[colorlinks,pdfpagelabels,pdfstartview = FitH,bookmarksopen = true,bookmarksnumbered = true,urlcolor=url,linkcolor = formula,plainpages = false,hypertexnames = false,citecolor = citation] {hyperref}

\def\dx{\,{\rm d}x}
\def\dy{\,{\rm d}y}

\def\dist{\,{\rm dist}}

\def\diam{\,{\rm diam}}

\allowdisplaybreaks
\makeatletter
\DeclareRobustCommand*{\bfseries}{%
  \not@math@alphabet\bfseries\mathbf
  \fontseries\bfdefault\selectfont
  \boldmath
}

\DeclareMathOperator*{\osc}{osc}

\makeatother

\newlength{\defbaselineskip}
\newcommand{\setlinespacing}[1]
           {\setlength{\baselineskip}{#1 \defbaselineskip}}

\newtheorem{theorem}{Theorem}

\newtheorem{definition}{Definition}
\newtheorem{remark}{Remark}[section]
\newtheorem{lemma}{Lemma}[section]
\newtheorem{proposition}{Proposition}[section]
\newcommand{\N}{\mathbb{N}}
\numberwithin{equation}{section}
\allowdisplaybreaks[4]

\newcommand{\rr}{\varrho}
\newcommand{\snr}[1]{\lvert #1\rvert}
\newcommand{\nr}[1]{\lVert #1 \rVert}

\title[%{\color{magenta} ??? Obstacle problem and ??? }
Fully nonlinear elliptic equations with non-homogeneous degeneracy]{Regularity for solutions of fully nonlinear elliptic equations with non-homogeneous degeneracy
}
\author{Cristiana De Filippis}  \address{Cristiana De Filippis\\Mathematical Institute, University of Oxford\\ Andrew Wiles Building, Radcliffe Observatory Quarter, Woodstock Road, Oxford, OX26GG, Oxford, United Kingdom} \email{\texttt{Cristiana.DeFilippis@maths.ox.ac.uk}}

\begin{document}

\subjclass[2010]{35J60, 35J70\vspace{1mm}} %%ALERT CHECK 35J60 23J70 35B65 35D40

\keywords{Fully nonlinear degenerate equations, Double Phase problems, Regularity\vspace{1mm}}

\thanks{{\it Acknowledgements.}\ The author is supported by the Engineering and Physical Sciences Research Council (EPSRC): CDT Grant Ref. EP/L015811/1. 
\vspace{1mm}}

\maketitle

\begin{abstract}
We prove that viscosity solutions to fully nonlinear elliptic equations with degeneracy of double phase type are locally $C^{1,\gamma}$-regular.
\end{abstract}
\vspace{3mm}
{\small \tableofcontents}

\setlinespacing{1.08}

\newcommand{\R}{\mathds{R}^n}

\section{Introduction}\label{intro}
We prove $C^{1,\gamma}$-local regularity for viscosity solutions of problem
\begin{flalign}\label{eq}
\left[\snr{Du}^{p}+a(x)\snr{Du}^{q}\right]F(D^{2}u)=f(x) \ \ \mbox{in} \ \ \Omega, \quad 0 \leq a(\cdot)\in C^0(\Omega)\,,\quad  0< p \leq q, 
\end{flalign}
where $\Omega\subset \mathbb{R}^{n}$, $n\ge 2$ is an open and bounded domain. 
Equation \eqref{eq} is a new model of singular fully nonlinear elliptic equation featuring an inhomogeneous degenerate term modelled upon the double phase integrand
\begin{flalign}\label{occurring}
H(x,z):=\left[\snr{z}^{p}+a(x)\snr{z}^{q}\right] \,,\quad 0\,,\qquad  1< p < q.
\end{flalign}
Introduced in the variational setting by V. V. Zhikov \cite{z1,z2,z3} in order to study homogeneization model problems and the occurrence of Lavrentiev phenomenon, functionals of type
\begin{flalign}\label{vardp}
w\mapsto \int H(x,Dw) \ \dx
\end{flalign}
are a particular instance of variational integrals with $(p,q)$-growth, first studied by Marcellini in \cite{ma1,ma2}. They are relevant in Materials Science since they can be used to describe the behaviour of strongly anisotropic materials whose hardening properties, linked to the gradient growth exponent, change with the point. In particular, a mixture of two different materials, with hardening exponents $p$ and $q$ respectively, can be realized according to the geometry dictated by the zero set of the coefficient $a(\cdot)$, i.e., $\{a(x)=0\}$. More details on this point can be found in \cite{comi}. The regularity theory for minimizers of \eqref{vardp} attracted lots of attention recently. We refer to \cite{bacomiha, bacomi,comi,comib} for a rather comprehensive account on the matter and for an explanation of the peculiar problems occurring when considering mixed degenerate structures as the one arising from \eqref{occurring}. For example, connections with Harmonic Analysis, initially established in \cite{comi, comib}, have been exploited in \cite{hastojfa}. Linkages with interpolation methods \cite{hastotams}, and Calder\'on-Zygmund estimates \cite{comicz, demicz}, have also been established, while, on a more applied sides, applications to image restorations problems have been recently given \cite{hastok}. See also \cite{chde} for the obstacle problem and some potential theoretic considerations, \cite{demi} for the manifold constrained case and \cite{depa} for the regularity features of viscosity solutions of equations related to the fractional Double-Phase integral
\begin{flalign*}
w \mapsto \int \frac{\snr{w(x)-w(y)}^{p}}{\snr{x-y}^{n+sp}}+a(x,y)\frac{\snr{w(x)-w(y)}^{q}}{\snr{x-y}^{n+tq}}\dy.
\end{flalign*}
This last paper is particularly important in our setting as it provides another instance of the basic regularity assumptions we are going to consider here; see comments after Theorem \ref{mor}. 
 
Our result brings the double phase energy into the realm of fully nonlinear elliptic equations, under sharp assumptions. Precisely, we prove the following:
\begin{theorem}\label{reg}
Under assumptions \eqref{ell} and \eqref{assf}-\eqref{assff}, let $u\in C(\Omega)$ be a viscosity solution of problem \eqref{eq}. Then there exists $\gamma=\gamma(n,\lambda,\Lambda,p)\in (0,1)$ such that $u\in C^{1,\gamma}_{loc}(\Omega)$ and, if $U\Subset \Omega$ is any open set there holds 
\begin{flalign}\label{essst}
[u]_{1+\gamma;U}\le c\left(1+\nr{u}_{L^{\infty}(\Omega)}+\nr{f}_{L^{\infty}(\Omega)}^{\frac{1}{p+1}}\right),
\end{flalign}
with $c=c(n,\lambda,\Lambda,p,q, \dist(U,\partial \Omega))$.
\end{theorem}
The outcome of Theorem \ref{reg} is sharp, in the light of the observation made in \cite[Example 1]{is}, which is consistent with our case when $a(\cdot)\equiv 0$. An important step towards the proof of Theorem \ref{reg}, consists in showing that normalized viscosity solutions of a suitable switched version of problem \eqref{eq} are $\beta_{0}$-H\"older continuous for some $\beta_{0} \in (0,1)$, i.e., 
\begin{theorem}\label{mor}Under assumptions \eqref{ell} and \eqref{assf}-\eqref{assff}, let $\bar{\xi}\in \mathbb{R}^{n}$ be an arbitrary vector and $\bar{u}\in C(B_{1})$ a normalized viscosity solution of
\begin{flalign}\label{20}
\left[\snr{\bar{\xi}+D\bar{u}}^{p}+\bar{a}(x)\snr{\bar{\xi}+D\bar{u}}^{q}\right]\bar{F}(D^{2}\bar{u})=\bar{f}(x) \ \ \mbox{in} \ \ B_{1}.
\end{flalign}
Then $\bar{u}\in C^{0,\beta_{0}}_{loc}(B_{1})$ for some $\beta_{0}\in (0,1)$ and if $B_{\rr}\subset B_{1}$ is any ball, there holds that
\begin{flalign}\label{mor1}
[\bar{u}]_{0,\beta_{0};B_{\rr}}\le c(n,\lambda,\Lambda,p,\rr).
\end{flalign}
\end{theorem}
We refer to Section \ref{small} for the precise definition of the various quantities involved in the previous statement. Theorem \ref{mor} provides a first compactness result for solutions of \eqref{eq}, which in turn will be fundamental in proving a $\bar{F}$-harmonic approximation lemma, crucial for transferring the regularity from solutions of the homogeneous equation   $$\bar{F}(D^{2}w)=0\quad \mbox{in} \ \ B_{1}$$ 
to solutions of \eqref{eq}. An interesting phenomenon is revealed in Theorem \ref{reg}: in sharp contrast to what happens in the variational setting \cite{bacomi,comi,comib}, where a quantitative H\"older continuity (depending on $p,q$) of $a(\cdot)$ is needed to get regular minima \cite{eslemi,fomami}, here the plain continuity of $a(\cdot)$ suffices. This is accordance to what found in the fractional viscosity setting \cite{depa} (see also \cite{lind}). In fact, to prove our regularity results, we just ask that the coefficient $a(\cdot)$ is continuous and no restriction on the size of the difference $0\le q-p$ is imposed. This makes Theorem \ref{reg} sharp from the viscosity theory viewpoint.
\\
Equation \eqref{eq} is an example of singular fully nonlinear elliptic equations, whose most celebrated prototype is
\begin{flalign}\label{classico}
\snr{Du}^{p}F(D^{2}u)=f\quad \mbox{in} \ \ B_{1},
\end{flalign}
see e.g. \cite{is}. Structures of this type, in the setting of viscosity solutions, often occurs in the theory of stochastic games \cite{apr,bamu}. Several aspects of this class of partial differential equations have already been investigated: comparison principle and Liouville-type theorems \cite{bd1}, properties of eigenvalues and eigenfunctions \cite{bd4,bd5}, Alexandrov-Bakelman-Pucci estimates \cite{dafequ,imb}, Harnack inequalities \cite{dafequ1,imb} and regularity \cite{bd2,bd3,is}. In particular, in order to study possible anisotropic problems, in \cite{bprt} the variable exponent case for the degeneracy is analyzed: precisely, it is shown that viscosity solutions of equations modelled on
\begin{flalign*}
\snr{Du}^{p(x)}F(D^{2}u)=f(x)\quad \mbox{in} \ \ B_{1}
\end{flalign*}
have H\"older continuous gradient. As carefully explained in \cite{bacomiha}, there is a sort of borderline structure between the classical case \label{classico} and a genuinely anisotropic case, which is given in fact by the double phase case. It is indeed the aim of this paper to treat such a fully anisotropic case. We yet notice that related anisotropic structures involving milder transitions are yet considered in the framework of stochastic tug-of-war games in \cite{arr, sita}.\\ 
The paper is organized as follows: in Section \ref{pre} we describe our framework, fully detail the problem and list the main assumptions we adopt. In section \ref{hol} we first explain how to reduce the problem to a smallness regime, then prove that normalized viscosity solutions of a switched version of \eqref{eq} are H\"older continuous. Finally, Section \ref{locreg} is devoted to the proof of Theorem \ref{reg} which crucially relies on a compactness argument leading to the construction, via an iterative procedure, of a uniform modulus of continuity of the difference between the solution and a suitably rescaled plane.
\section{Preliminaries}\label{pre}
We shall split this section in three parts: first, we display our notation, then we collect the main assumptions governing problem \eqref{eq}, and finally we report some well-known results on the theory of viscosity solutions to uniformly elliptic operators. 
\subsection{Notation}
In this paper, $\Omega\subset \mathbb{R}^{n}$, $n\ge 2$ is an open and bounded domain, the open ball of $\mathbb{R}^{n}$ centered at $x_{0}$ with positive radius $\rr$ is denoted by $B_{\rr}(x_{0}):=\left\{x\in \mathbb{R}^{n}\colon \snr{x-x_{0}}<\rr\right\}$. When not relevant, or clear from the context, we will omit indicating the center, $B_{\rr}\equiv B_{\rr}(x_{0})$. In particular, for $\rr=1$ and $x_{0}= 0$, we shall simply denote $B_{1}\equiv B_{1}(0)$. With $\mathcal{S}(n)$ we mean the space of $n\times n$ symmetric matrices. As usual, we denote by $c$ a general constant larger than one. Different occurrences from line to line will be still indicated by $c$ and relevant dependencies from certain parameters will be emphasized using brackets, i.e.: $c(n,p)$ means that $c$ depends on $n$ and $p$. For $g\colon B_{1}\to \mathbb{R}^{k}$ and $U\subset B_{1}$, with $\beta \in (0,1]$ being a given number we shall denote
\begin{flalign*}
[g]_{0,\beta;U}:=\sup_{x,y \in U; x\not=y}\frac{\snr{g(x)-g(y)}}{\snr{x-y}^{\beta}}, \qquad [g]_{0,\beta}:=[g]_{0,\beta;B_{1}}.
\end{flalign*}
It is well known that the quantity defined above is a seminorm and when $[g]_{0,\beta;U}<\infty$, we will say that $g$ belongs to the H\"older space $C^{0,\beta}(U,\mathbb{R}^{k})$. Furthermore, $g\in C^{1,\beta}(U,\mathbb{R}^{k})$ provided that
\begin{flalign*}
[g]_{1+\beta;U}:=\sup_{\rr>0,x\in U}\inf_{\xi\in \mathbb{R}^{n},\kappa\in \mathbb{R}}\sup_{y\in B_{\rr}(x)\cap U}\rr^{-(1+\beta)}\snr{g(y)-\xi\cdot y-\kappa}<\infty.
\end{flalign*}
Finally, given any $n\times n$ matrix $A$, with $\texttt{tr}(A)$ we will denote the trace of $A$, i.e., the sum of all its eigenvalues, by $\texttt{tr}(A^{+})$ the sum of all positive eigenvalues of $A$ and by $\texttt{tr}(A^{-})$ the sum of all negative eigenvalues of $A$.

\subsection{On uniformly elliptic operators}\label{uniop} A continuous map $G\colon \Omega\times \mathbb{R}^{n}\times \mathcal{S}(n)\to \mathbb{R}$ is monotone if
\begin{flalign}\label{mon}
G(x,z,M)\le G(x,z,N)\quad \mbox{whenever} \ \ M,N\in \mathcal{S}(n) \ \ \mbox{satisfy} \ \ M\ge N.
\end{flalign}

The $(\lambda,\Lambda)$-ellipticity condition for an operator $F\colon \mathcal{S}(n)\to \mathbb{R}$ prescribes that, whenever $A,B\in \mathbb{S}(n)$ are symmetric matrices with $B\ge 0$,
\begin{flalign}\label{ell}
\lambda \texttt{tr}(B)\le F(A)-F(A+B)\le \Lambda \texttt{tr}(B)
\end{flalign}
for and some fixed constants $0<\lambda\le \Lambda$. As stressed in \cite{is}, under this definition $F(A):=-\texttt{tr}(A)$ is uniformly elliptic with $\lambda=\Lambda=1$, so the usual Laplace operator is uniformly elliptic. Moreover, it is easy to see that, if $L$ is any fixed, positive constant, then the operator $F_{L}(M):=LF\left(\frac{1}{L}M\right)$ satisfies \eqref{ell} with the same constants $0<\lambda\le\Lambda$. Moreover, \eqref{ell} is also verified by the operator $\tilde{F}(M):=-F(-M)$, in fact, if for $A,B\in \mathcal{S}(n)$ with $B\ge 0$, we set $A_{1}:=-A-B$ we immediately see that
\begin{flalign*}
&\tilde{F}(A)-\tilde{F}(A+B)=F(A_{1})-F(A_{1}+B)\stackrel{\eqref{ell}}{\ge}\lambda \texttt{tr}(B),\nonumber\\ &\tilde{F}(A+B)-\tilde{F}(A)=F(A_{1})-F(A_{1}+B)\stackrel{\eqref{ell}}{\le}\Lambda \texttt{tr}(B).
\end{flalign*}
In the framework of $(\lambda,\Lambda)$-elliptic operators, important concepts are the so-called Pucci extremal operators $\mathcal{M}^{\pm}_{\lambda,\Lambda}$, which are, respectively, the maximum and the minimum of all the uniformly elliptic functions $F$ with $F(0)=0$. In particular they admit the following compact form
\begin{flalign}\label{m-}
\mathcal{M}^{+}_{\lambda,\Lambda}(A)= -\Lambda \texttt{tr}(A^{-})-\lambda \texttt{tr}(A^{+})\quad \mbox{and}\quad \mathcal{M}^{-}_{\lambda,\Lambda}(A)=-\Lambda \texttt{tr}(A^{+})-\lambda \texttt{tr}(A^{-}).
\end{flalign}
With the Pucci operators at hand, we can reformulate \eqref{ell} as 
\begin{flalign}\label{elll}
\mathcal{M}^{-}_{\lambda,\Lambda}(B)\le F(A+B)-F(A)\le \mathcal{M}^{+}_{\lambda,\Lambda}(B),
\end{flalign}
for all $A,B\in \mathcal{S}(n)$. Next, we turn our attention to equation
\begin{flalign}\label{211r}
G_{\xi}(x,Du,D^{2}u):=G(x,\xi+Du,D^{2}u)=0\quad \mbox{in} \ \ \Omega,
\end{flalign}
with $G$ continuous and satisfying \eqref{mon} and $\xi \in \mathbb{R}^{n}$ arbitrary vector. The concept of viscosity solution to \eqref{211r} can be explained as follows:
\begin{definition}\emph{\cite{baim}}\label{def}
A lower semicontinuous function $v$ is a viscosity supersolution of \eqref{211r} if whenever $\varphi\in C^{2}(\Omega)$ and $x_{0}\in \Omega$ is a local minimum point of $v-\varphi$, then 
\begin{flalign*}
G_{\xi}(x_{0},D\varphi(x_{0}),D^{2}\varphi(x_{0}))\ge 0,
\end{flalign*}
while an upper semicontinuous function $w$ is a viscosity subsolution to \eqref{211r} provided that if $x_{0}$ is a local maximum point of $w-\varphi$, there holds
\begin{flalign*}
G_{\xi}(x_{0},D\varphi(x_{0}),D^{2}\varphi(x_{0}))\le 0.
\end{flalign*}
The map $u\in C(\Omega)$ is a viscosity solution of \eqref{211r} if it is a the same time a viscosity subsolution and a viscosity supersolution. 
\end{definition}
Another important notion is the one of subjets and superjets.
\begin{definition}\label{def2}\emph{\cite{baim}}
Let $v\colon \Omega\to \mathbb{R}$ be an upper semicontinuous function and $w\colon \Omega\to \mathbb{R}$ be a lower semicontinuous function.
\begin{itemize}
    \item A couple $(z,X)\in \mathbb{R}^{n}\times \mathcal{S}(n)$ is a superjet of $v$ at $x\in \Omega$ if
    \begin{flalign*}
    v(x+y)\le v(x)+z\cdot y+\frac{1}{2}Xy\cdot y+o(\snr{y}^{2}).
    \end{flalign*}
    \item A couple $(z,X)\in \mathbb{R}^{n}\times \mathcal{S}(n)$ is a subjet of $w$ at $x\in \Omega$ if
    \begin{flalign*}
    w(x+y)\ge w(x)+z\cdot y+\frac{1}{2}Xy\cdot y+o(\snr{y}^{2}).
    \end{flalign*}
    \item A couple $(z,X)\in \mathbb{R}^{n}\times \mathcal{S}(n)$ is a limiting superjet of $v$ ar $x\in \Omega$ if there exists a sequence $\{x_{j},z_{j},X_{j}\}\to_{j\to \infty}\{x,z,X\}$ such that $\{z_{j},X_{j}\}$ is a superjet of $v$ at $x_{j}$ and $v(x_{j})\to_{j\to \infty}v(x)$.
    \item A couple $(z,X)\in \mathbb{R}^{n}\times \mathcal{S}(n)$ is a limiting subjet of $w$ ar $x\in \Omega$ if there exists a sequence $\{x_{j},z_{j},X_{j}\}\to_{j\to \infty}\{x,z,X\}$ such that $\{z_{j},X_{j}\}$ is a subjet of $w$ at $x_{j}$ and $w(x_{j})\to_{j\to \infty}w(x)$.
\end{itemize}
\end{definition}
Now we are in position to present a variation on the celebrated Ishii-Lions lemma, \cite{cil}.
\begin{proposition}\label{p1}\emph{\cite{baim}}
Let $v$ be an upper semicontinuous viscosity subsolution of \eqref{211r}, $w$ a lower semicontinuous viscosity supersolution of \eqref{211r}, $U\Subset \Omega$ an open set and $\psi \in C^{2}(U\times U)$. If $(\bar{x},\bar{y})\in U\times U$ is a local maximum point of $v(x)-w(y)-\psi(x,y)$, then, for any $\iota>0$ there exists a threshold $\hat{\delta}=\hat{\delta}(\iota, \nr{D^{2}\psi})>0$ such that for all $\delta \in (0,\hat{\delta})$ we have matrices $X_{\delta},Y_{\delta}\in \mathcal{S}(n)$ such that
\begin{flalign*}
G_{\xi}(\bar{x},v(\bar{x}),\partial_{x}\psi(\bar{x},\bar{y}),X_{\delta})\le 0\le G_{\xi}(\bar{y},w(\bar{y}),-\partial_{y}\psi(\bar{x},\bar{y}),Y_{\delta}).
\end{flalign*}
and the inequality
\begin{flalign*}
-\frac{1}{\delta}\mbox{Id}\le\begin{bmatrix} X_{\delta} & 0 \\ 0 & -Y_{\delta}\end{bmatrix}\le D^{2}\psi(\bar{x},\bar{y})+\delta\mbox{Id}
\end{flalign*}
holds true.
\end{proposition}
\begin{remark}
\emph{In \cite{baim} condition \eqref{mon} appears as an "ellipticity assumption". We shall refer to it as "monotonicity" to avoid any confusion with \eqref{ell}.}
\end{remark}
\subsection{Main assumptions}\label{mass}
When dealing with problems \eqref{eq}-\eqref{20}, the following assumptions will be in force. As mentioned before, the set $\Omega\subset \mathbb{R}^{n}$ is an open and bounded domain. Up to dilations and translations, there is no loss of generality in assuming that $B_{2}\Subset \Omega$. The nonlinear operator $F$ is continuous and $(\lambda,\Lambda)$-elliptic in the sense of \eqref{ell}. Moreover
\begin{flalign}\label{assf}
F\in C(\mathcal{S}(n),\mathbb{R}), \quad F(0)=0.
\end{flalign}
Concerning the non-homogeneous degeneracy term appearing in \eqref{eq}, we shall ask that the exponents $p,q$ and the modulating coefficient $a(\cdot)$ are so that
\begin{flalign}\label{assh}
0\le a(\cdot)\in C(\Omega)\quad \mbox{and} \quad 0<p\le q.
\end{flalign}
Finally, the forcing term $f$ verifies
\begin{flalign}\label{assff}
f\in C(\Omega).
\end{flalign}
All the facts exposed in Section \ref{uniop} easily adjust to our double phase setting simply choosing 
\begin{flalign*}
G_{\xi}(x,z,M)=\left[\left(\snr{\xi+z}^{p}+a(x)\snr{\xi+z}^{q}\right)F(M)-f(x)\right]\quad \mbox{for all} \ \ (x,z,M)\in \Omega\times \mathbb{R}^{n}\times \mathcal{S}(n).
\end{flalign*}
By \eqref{assf} and \eqref{assff} we then know that $G\in C( \Omega\times \mathbb{R}^{n}\times \mathcal{S}(n),\mathbb{R})$. Moreover, the $(\lambda,\Lambda)$-ellipticity of $F$ guarantees that $G_{\xi}$ satisfies \eqref{mon}, so, in particular Definitions \ref{def}-\ref{def2} and Proposition \ref{p1} are available to us.
\begin{remark}
\emph{%A few comments are in order. 
In $\eqref{assh}_{1}$ we assumed that the modulating coefficient $a(\cdot)$ is continuous, but all the results proved in this paper hold \emph{verbatim} if we only take $a(\cdot)$ bounded and defined everywhere. This makes Theorem \ref{reg} a solid blueprint for studying further problems in which $a(\cdot)$ depends also from $u$ and is allowed to have discontinuities.}
%Moreover, we can safely conjecture that techniques analogous to those presented in this paper extend the result in Theorem \ref{reg} to viscosity solutions of fully nonlinear degenerate elliptic equations whose degeneracy term presents an arbitrary number of phases, i.e.:}
%\begin{flalign*}
%&\left[\snr{Du}^{p}+\sum_{j=1}^{k}a_{j}(x)\snr{Du}^{p_{j}}\right]F(D^{2}u)=f(x)\quad \mbox{in} \ \ \Omega\\
%&0\le a_{j}(\cdot)\in C(\Omega),\quad 0<p\le p_{1}\le \cdots\le p_{k},\quad f\in C(\Omega),
%\end{flalign*}
%\emph{see \cite{deoh} for the variational counterpart of this matter.}
\end{remark}
\subsection{The homogeneous problem}
Viscosity solutions of the homogeneous problem
\begin{flalign}\label{probhom}
F(D^{2}v)=0\quad \mbox{in} \ \ B_{1}
\end{flalign}
will have a crucial role in the proof of the main results of this paper.
\begin{definition}\label{dhar}
Let $F$ be as in \eqref{ell}-\eqref{assf}. A function $h \in C(U)$ is said to be $F$-harmonic in $B_{1}$ if it is a viscosity solution of \eqref{probhom}.
\end{definition}
As one could expect, maps as in Definition \ref{dhar} have good regularity properties, as the next proposition shows. For a proof, we refer to \cite[Corollary 5.7]{caca}.
\begin{proposition}\label{rhar}\emph{\cite{caca}}
Let $F$ be as in \eqref{ell}-\eqref{assf} and $h\in C(B_{1})$ be a viscosity solution of \eqref{probhom}. Then, there exist $\alpha=\alpha(n,\lambda,\Lambda)\in (0,1)$ and $c=c(n,\lambda,\Lambda)>0$ such that
\begin{flalign}\label{031}
\nr{h}_{C^{1,\alpha}(\bar{B}_{1/2})}\le c\nr{h}_{L^{\infty}(B_{1})}.
\end{flalign}
\end{proposition}
\begin{remark}
\emph{Proposition \ref{rhar} in particular states that if $h\in C(B_{1})$ is $F$-harmonic, then it is $C^{1,\alpha}$ around zero, which means that for all $\rr\in (0,1)$ there exists a $\xi_{\rr}\in \mathbb{R}^{n}$ such that}
\begin{flalign}\label{030}
\osc_{B_{\rr}}(h-\xi_{\rr}\cdot x)\le c(n,\lambda,\Lambda)\rr^{1+\alpha}.
\end{flalign}
\emph{Now fix $\sigma\in (0,1)$ be so small that}
\begin{flalign}\label{delta}
c\sigma^{\alpha}<\frac{1}{4},
\end{flalign}
\emph{where $c=c(n,\lambda,\Lambda)$ is the constant appearing in \eqref{031} and let $\xi_{\sigma}\in \mathbb{R}^{n}$ be the corresponding vector in \eqref{030}. According to the choice in \eqref{delta}, \eqref{030} reads as}
\begin{flalign}\label{0331}
\osc_{B_{\sigma}}(h-\xi_{\sigma}\cdot x)\le \frac{1}{4}\sigma, \quad \mbox{with}\quad \sigma=\sigma(n,\lambda,\Lambda).
\end{flalign}
\emph{This will be helpful later on}.
\end{remark}
\section{$\beta_{0}$-H\"older continuity}\label{hol}
In this section we will prove that normalized viscosity solutions of problem
\begin{flalign}\label{shift}
\left[\snr{\xi+Du}^{p}+a(x)\snr{\xi+Du}^{q}\right]F(D^{2}u)=f(x)\quad \mbox{in} \ \ \Omega
\end{flalign}
where $\xi\in \mathbb{R}^{n}$ is any vector, are locally $\beta_{0}$-H\"older continuous for some $\beta_{0}\in (0,1)$. A direct consequence of this, is equicontinuity for sequences of normalized viscosity solutions to certain problems of the type \eqref{eq}, see the proof of Lemma \ref{har} in Section \ref{locreg}.
\subsection{Smallness regime}\label{small}
In this part, we use the scaling features of the shifted operator in \eqref{shift} to trace the problem back to a smallness regime. In other terms, we blow and scale $u$ in order to construct another map $\bar{u}$, solution in $B_{1}$ of a problem having the same structure as \eqref{shift}, and such that, for a given $\varepsilon>0$,
\begin{flalign}\label{sreg}
\osc_{B_{1}}(\bar{u})\le 1\quad \mbox{and} \quad \nr{\bar{f}}_{L^{\infty}(B_{1})}\le \varepsilon,
\end{flalign}
where $\bar{f}$ is a suitable modified version of the forcing term appearing in \eqref{shift}. Under these conditions, $\bar{u}$ is called "normalized viscosity solution". Let us show this construction. Set
\begin{flalign}\label{k}
K:=2\left(1+\nr{u}_{L^{\infty}(\Omega)}+\nr{f}_{L^{\infty}(\Omega)}^{\frac{1}{p+1}}\right)
\end{flalign}
and let 
\begin{flalign}\label{mm}
m\in \left(0,\min\left\{1,\frac{\diam(\Omega)}{4}\right\}\right)
\end{flalign}
be a constant whose size will be quantified later on. Notice that, if $u\in C(\Omega)$ is a viscosity solution to \eqref{shift} and $x_{0}\in \Omega$, then clearly $u$ is a continuous viscosity solution of \eqref{shift} in $B_{m}(x_{0})$, with $m$ as in \eqref{mm}. Now, for $x\in B_{1}$, $M\in \mathcal{S}(n)$, $K$ and $m$ as in \eqref{k}-\eqref{mm} respectively, define the following quantities:
\begin{flalign*}
&\bar{u}(x):=\frac{u(x_{0}+mx)}{K}, \quad \bar{a}(x):=\left(\frac{K}{m}\right)^{q-p}a(x_{0}+mx),\\
&\bar{F}(M):=\frac{m^{2}}{K}F\left(\frac{K}{m^{2}}M\right), \quad \bar{f}(x):=\frac{m^{p+2}}{K^{p+1}}f(x_{0}+mx).
\end{flalign*}
Since $u$ is a viscosity solution to \eqref{shift} in $B_{m}(x_{0})$, it is easy to see that $\bar{u}$ is a viscosity solution of
\begin{flalign*}
\left[\snr{\bar{\xi}+D\bar{u}}^{p}+\bar{a}(x)\snr{\bar{\xi}+D\bar{u}}^{q}\right]\bar{F}(D^{2}\bar{u})=\bar{f}(x) \quad \mbox{in} \ \ B_{1},
\end{flalign*}
where $\bar{\xi}:=(m/K)\xi$. In particular, when $\xi=0$, we have \eqref{eq} in smallness regime 
\begin{flalign}\label{eqs}
\left[\snr{D\bar{u}}^{p}+\bar{a}(x)\snr{D\bar{u}}^{q}\right]\bar{F}(D^{2}\bar{u})=\bar{f}(x)\quad \mbox{in} \ \ B_{1}.
\end{flalign}
By definition, there holds
\begin{flalign}\label{19}
\osc_{B_{1}}(\bar{u})\le 1,\quad \nr{\bar{a}}_{L^{\infty}(B_{1})}\le \left(\frac{K}{m}\right)^{q-p}\nr{a}_{L^{\infty}(\Omega)}, \quad \nr{\bar{f}}_{L^{\infty}(B_{1})}\le m^{p+2},
\end{flalign}
moreover, a quick computation shows that, since $F$ satisfies \eqref{ell}, $\bar{F}$ is $(\lambda,\Lambda)$-elliptic as well. Now, for arbitrary $\varepsilon>0$, we fix $m=\varepsilon^{\frac{1}{p+2}}$, it follows that $\nr{\bar{f}}_{L^{\infty}(B_{1})}\le \varepsilon$, therefore $\bar{u}$ is in smallness regime. Finally, notice that if $\bar{u}$ is a solution of equation \eqref{20}, then $\bar{u}_{c}:=\bar{u}+c$ for any $c\in \mathbb{R}$ is a solution as well, so there is no loss of generality in taking $\bar{u}(0)=0$. We will assume this throughout the paper.
\begin{remark}
\emph{Owing to the non-homogeneity of problem \eqref{eq}, the scaling factor $m$ appears also in the expression of $\bar{a}$, thus leading to the bound $\eqref{19}_{2}$. As we shall see, $m$ will never influence the constants appearing in the forthcoming estimates and it will ultimately depend only from $(n,\lambda,\Lambda,p,q)$.}
\end{remark}
\begin{remark}\label{rem1}
\emph{Clearly, it is enough to prove Theorem \ref{reg} for $\bar{u}\in C(B_{1})$ solution of \eqref{20}. In fact, as soon as we know that}
\begin{flalign*}
[\bar{u}]_{1+\gamma;B_{1/2}}\le c(n,\lambda,\Lambda,p,q),
\end{flalign*}
\emph{then, from the definitions given before, it directly follows that, after scaling,}
\begin{flalign*}
[u]_{1+\gamma;B_{m/2}(x_{0})}\le \frac{cK}{m^{1+\gamma}}\le c(n,\lambda,\Lambda,p,q)\left(1+\nr{u}_{L^{\infty}(\Omega)}+\nr{f}_{L^{\infty}(\Omega)}^{\frac{1}{p+1}}\right),
\end{flalign*}
\emph{so, after a standard covering argument we can conclude that $u\in C^{1,\gamma}_{loc}(\Omega)$ as stated in Theorem \ref{reg} and \eqref{essst} directly follows via a standard covering argument}.
\end{remark}
\subsection{Proof of Theorem \ref{mor}}
The proof of Theorem \ref{mor} is a direct consequence of Proposition \ref{ex} below.
\begin{proposition}\label{ex}
Let $\bar{u}\in C(B_{1})$ be a normalized viscosity solution of \eqref{20} and $\tilde{x}\in B_{1}$ be any point with $B_{1/2}(\tilde{x})\Subset B_{1}$. There exist a positive number $s_{0}=s_{0}(n,\lambda,\Lambda,p,\rr)$ and a threshold $\rr_{*}\in \left(0,\frac{1}{2}\right)$ such that if $\rr\in \left(0,\rr_{*}\right)$, then
\begin{itemize}
    \item $\snr{\bar{\xi}}>s_{0}^{-1}\Rightarrow\bar{u}\in \mbox{Lip}_{loc}(B_{\rr/2})$;
    \item $\snr{\bar{\xi}}\le s_{0}^{-1}\Rightarrow\bar{u}\in C^{0,\beta_{0}}_{loc}(B_{\rr/2})$, for some $\beta_{0}\in (0,1)$.
\end{itemize}
\begin{proof}
We fix the threshold 
\begin{flalign}\label{rr}
\rr_{*}:=\frac{1}{20000}
\end{flalign}
and take $\rr\in \left(0,\rr_{*}\right)$. We shall prove that there are two constants $A_{1}=A_{1}(n,\lambda,\Lambda,p,\rr)$ and $A_{2}=A_{2}(\rr)$ such that
\begin{flalign}\label{m}
\mathcal{L}(\tilde{x}):=\sup_{x,y\in B_{\rr}}\left(\bar{u}(x)-\bar{u}(y)-A_{1}\omega(\snr{x-y})-A_{2}\left(\snr{x-\tilde{x}}^{2}+\snr{y-\tilde{x}}^{2}\right)\right)\le 0, 
\end{flalign}
for all $\tilde{x}\in B_{\rr/2}$. In \eqref{m},
\begin{flalign}\label{omega}
\omega(t):=t^{\beta_{0}} \ \ \mbox{if} \ \ \snr{\bar{\xi}}\le s_{0}^{-1},\quad \omega(t):=\begin{cases}\ t-\omega_{0}t^{3/2}\ \ &\mbox{if} \ \ t\le t_{0}\\
\ \omega(t_{0}) \ \ &\mbox{if} \ \ t>t_{0} 
\end{cases}\ \ \mbox{if} \ \ \snr{\bar{\xi}}>s_{0}^{-1},
\end{flalign}
where
\begin{flalign}\label{beta0}
\beta_{0}\in \left(\frac{1}{4},\frac{1}{2}\right)
\end{flalign}
is any number, $t_{0}:=\left(\frac{2}{3\omega_{0}}\right)^{2}$, $\omega_{0}\in \left(0,\frac{2}{3}\right)$ is such that $t_{0}\ge 1$ and 
\begin{flalign}\label{s0}
s_{0}:=\frac{1}{16(A_{1}+A_{2}+1)}.
\end{flalign}
For reasons that will be clear in a few lines, we set
\begin{flalign}\label{a1}
A_{1}:=\frac{(2A_{2}+1)\left[80+2^{4(p+1)}\left(\Lambda(n-1)+\lambda+1\right)\right]}{\omega_{0}\min\{\lambda,1\}}
\end{flalign}
and
\begin{flalign}\label{a2}
A_{2}:=(4/\rr)^{2}.
\end{flalign}
By contradiction, let us assume that 
\begin{flalign}\label{contr}
\mbox{there exists} \ \ \tilde{x}\in B_{\rr/2} \ \ \mbox{such that} \ \ \mathcal{L}(\tilde{x})>0 \ \ \mbox{for all positive} \ \ A_{1},A_{2}. 
\end{flalign}
To reach a contradiction, we consider the auxiliary functions
\begin{flalign*}
\begin{cases}
\ \psi(x,y):=A_{1}\omega(\snr{x-y})+A_{2}\left(\snr{x-\tilde{x}}^{2}+\snr{y-\tilde{x}}^{2}\right)\\
\ \phi(x,y):=\bar{u}(x)-\bar{u}(y)-\psi(x,y).
\end{cases}
\end{flalign*}
Let $(\bar{x},\bar{y})\in \bar{B}_{\rr}\times \bar{B}_{\rr}$ is a point of maximum for $\phi$. By \eqref{contr}, $\phi(\bar{x},\bar{y})=\mathcal{L}(\tilde{x})>0$, thus
\begin{flalign*}
A_{1}\omega(\snr{\bar{x}-\bar{y}})+A_{2}\left(\snr{\bar{x}-\tilde{x}}^{2}+\snr{\bar{y}-\tilde{x}}^{2}\right)\le \bar{u}(\bar{x})-\bar{u}(\bar{y})\stackrel{\eqref{sreg}_{1}}{\le}1.
\end{flalign*}
The choice made in \eqref{a2} forces $\bar{x},\bar{y}$ to belong to the interior of $B_{\rr}$. In fact, plugging \eqref{a2} in the previous display, we get:
\begin{flalign}\label{small12}
\snr{\bar{x}}\le \snr{\bar{x}-\tilde{x}}+\snr{\tilde{x}}\le \frac{3\rr}{4}\quad \mbox{and}\quad \snr{\bar{y}}\le \snr{\bar{y}-\tilde{x}}+\snr{\tilde{x}}\le \frac{3\rr}{4}.
\end{flalign}
Moreover, $\bar{x}\not =\bar{y}$, otherwise $\mathcal{L}(\tilde{x})=\phi(\bar{x},\bar{y})=0$ and \eqref{m} would be immediately verified. This last observation clarifies that $\psi$ is smooth in a sufficiently small neighborhood of $(\bar{x},\bar{y})$, so the following position is meaningful:
\begin{flalign*}
&\bar{\xi}_{\bar{x}}:=\partial_{x}\psi(\bar{x},\bar{y})=A_{1}\omega'(\snr{\bar{x}-\bar{y}})\frac{\bar{x}-\bar{y}}{\snr{\bar{x}-\bar{y}}}+2A_{2}(\bar{x}-\tilde{x}),\\
&\bar{\xi}_{\bar{y}}:=-\partial_{y}\psi(\bar{x},\bar{y})=A_{1}\omega'(\snr{\bar{x}-\bar{y}})\frac{\bar{x}-\bar{y}}{\snr{\bar{x}-\bar{y}}}-2A_{2}(\bar{y}-\tilde{x}).
\end{flalign*}
All in all, $\phi$ attains its maximum in $(\bar{x},\bar{y})$ inside $B_{\rr}\times B_{\rr}$ and $\psi$ is smooth around $(\bar{x},\bar{y})$, so Proposition \ref{p1} applies: for any $\iota>0$ we can find a threshold $\hat{\delta}=\hat{\delta}(\iota,\nr{D^{2}\psi})$ such that for all $\delta \in (0,\hat{\delta})$ the couple $(\bar{\xi}_{\bar{x}},X_{\delta})$ is a limiting subjet of $\bar{u}$ at $\bar{x}$ and the couple $(\bar{\xi}_{\bar{y}},Y_{\delta})$ is a limiting superjet of $\bar{u}$ at $\bar{y}$ and the matrix inequality
\begin{flalign}\label{m2}
\begin{bmatrix}
X_{\delta} & 0 \\ 0 & -Y_{\delta} 
\end{bmatrix}\le \begin{bmatrix}
Z & -Z \\ -Z & Z 
\end{bmatrix}+(2A_{2}+\delta)\mbox{Id}
\end{flalign}
holds, where we set
\begin{flalign*}
Z:=&A_{1}(D^{2}\omega)(\snr{\bar{x}-\bar{y}})\nonumber \\
=&A_{1}\left[\frac{\omega'(\snr{\bar{x}-\bar{y}})}{\snr{\bar{x}-\bar{y}}}\mbox{Id}+\left(\omega''(\snr{\bar{x}-\bar{y}})-\frac{\omega'(\snr{\bar{x}-\bar{y}})}{\snr{\bar{x}-\bar{y}}}\right)\frac{(\bar{x}-\bar{y})\otimes (\bar{x}-\bar{y})}{\snr{\bar{x}-\bar{y}}^{2}}\right].
\end{flalign*}
We fix $\delta=\min\left\{1,\frac{\hat{\delta}}{4}\right\}$ and apply \eqref{m2} to vectors of the form $(z,z)\in \mathbb{R}^{2n}$, to obtain 
\begin{flalign*}
\langle(X_{\delta}-Y_{\delta})z,z \rangle\le (4A_{2}+2)\snr{z}^{2}.
\end{flalign*}
This means that 
\begin{flalign}\label{alle}
\mbox{all the eigenvalues of} \ \ X_{\delta}-Y_{\delta} \ \ \mbox{are less than or equal to} \ \ 2(2A_{2}+1).
\end{flalign}
In particular, applying \eqref{m2} to the vector $\bar{z}:=\left(\frac{\bar{x}-\bar{y}}{\snr{\bar{x}-\bar{y}}},\frac{\bar{y}-\bar{x}}{\snr{\bar{x}-\bar{y}}}\right)$, we get
\begin{flalign*}
\left \langle (X_{\delta}-Y_{\delta})\frac{\bar{x}-\bar{y}}{\snr{\bar{x}-\bar{y}}},\right.&\left.\frac{\bar{x}-\bar{y}}{\snr{\bar{x}-\bar{y}}} \right \rangle\le 2(2A_{2}+1)+4A_{1}\omega''(\snr{\bar{x}-\bar{y}}).
\end{flalign*}
This yields in particular that
\begin{flalign}\label{eneg}
\mbox{at least one eigenvalue of} \ \ X_{\delta}-Y_{\delta} \ \ \mbox{is less than} \ \ 2(2A_{2}+1)+4A_{1}\omega''(\snr{\bar{x}-\bar{y}}).
\end{flalign}
%therefore
%\begin{flalign}\label{107}
%\texttt{tr}(X_{\delta}-Y_{\delta})\le 2n(2A_{2}+1)+4A_{1}\omega''(\snr{\bar{x}-\bar{y}}).
%\end{flalign}
At this point we study separately the two cases $\snr{\bar{\xi}}>s_{0}^{-1}$ and $\snr{\bar{\xi}}\le s_{0}^{-1}$.\\\\
\emph{Case 1: $\snr{\bar{\xi}}>s_{0}^{-1}$}. Define
\begin{flalign*}
s:=\snr{\bar{\xi}}^{-1}\in (0,s_{0}), \quad \tilde{\xi}:=s\bar{\xi}, \quad \tilde{a}(x):=s^{p-q}\bar{a}(x), \quad \tilde{f}(x):=s^{p}\bar{f}(x)
\end{flalign*}
and rewrite \eqref{20} as 
\begin{flalign}\label{104}
\left[\snr{\tilde{\xi}+sD\bar{u}}^{p}+\tilde{a}(x)\snr{\tilde{\xi}+sD\bar{u}}^{q}\right]\bar{F}(D^{2}\bar{u})=\tilde{f}(x)\quad \mbox{in} \ \ B_{1}.
\end{flalign}
Notice that, by \eqref{sreg} and \eqref{a2}, $\nr{\tilde{f}}_{L^{\infty}(B_{1})}\le s_{0}^{p}\varepsilon<\varepsilon$. Moreover, expanding the expression of $\omega$ in \eqref{eneg} (keep $\eqref{omega}_{2}$ in mind) and recalling the choice we made in \eqref{a1}-\eqref{a2}, we see that
\begin{flalign*}
2(2A_{2}+1)+4A_{1}\omega''(\snr{\bar{x}-\bar{y}})\le 2(2A_{2}+1)-3A_{1}\omega_{0}<0,
\end{flalign*}
where we also used that $\snr{\bar{x}-\bar{y}}<1$. This means that at least one eigenvalue of $X_{\delta}-Y_{\delta}$ is negative, therefore by $\eqref{m-}_{2}$, \eqref{alle} and \eqref{eneg} we have
\begin{flalign}\label{visco3}
\mathcal{M}^{-}_{\lambda,\Lambda}(X_{\delta}-Y_{\delta})\ge -2(2A_{2}+1)\left[\Lambda(n-1)+\lambda\right]+3\lambda\omega_{0}A_{1}.
\end{flalign}
With $\bar{\xi}_{\bar{x}}$, $\bar{\xi}_{\bar{y}}$ computed before, we write the two viscosity inequalities
\begin{flalign}\label{visco}
\begin{cases}
\ \left[\snr{\tilde{\xi}+s\bar{\xi}_{\bar{x}}}^{p}+\tilde{a}(\bar{x})\snr{\tilde{\xi}+s\bar{\xi}_{\bar{x}}}^{q}\right]\bar{F}(X_{\delta})\le \tilde{f}(\bar{x})\\
\ \left[\snr{\tilde{\xi}+s\bar{\xi}_{\bar{y}}}^{p}+\tilde{a}(\bar{y})\snr{\tilde{\xi}+s\bar{\xi}_{\bar{y}}}^{q}\right]\bar{F}(Y_{\delta})\ge \tilde{f}(\bar{y}).
\end{cases}
\end{flalign}
The choice we made in \eqref{s0} then yields that
\begin{flalign}\label{visco1}
\min\{\snr{\tilde{\xi}+s\bar{\xi}_{\bar{x}}},\snr{\tilde{\xi}+s\bar{\xi}_{\bar{y}}}\}\ge \frac{1}{2}
\end{flalign}
and, recalling also \eqref{elll}, we see that
\begin{flalign}\label{visco2}
\bar{F}(X_{\delta})\ge \bar{F}(Y_{\delta})+\mathcal{M}^{-}_{\lambda,\Lambda}(X_{\delta}-Y_{\delta}).
\end{flalign}
Combining the inequalities in the previous display we eventually get
\begin{flalign*}
&\frac{\tilde{f}(\bar{x})}{\left[\snr{\tilde{\xi}+s\bar{\xi}_{\bar{x}}}^{p}+\tilde{a}(\bar{x})\snr{\tilde{\xi}+s\bar{\xi}_{\bar{x}}}^{q}\right]}\stackrel{\eqref{visco}_{1}}{\ge}\bar{F}(X_{\delta})\stackrel{\eqref{visco2}}{\ge}\bar{F}(Y_{\delta})+\mathcal{M}^{-}_{\lambda,\Lambda}(X_{\delta}-Y_{\delta})\nonumber \\
&\qquad \stackrel{\eqref{visco3},\eqref{visco}_{2}}{\ge}\frac{\tilde{f}(\bar{y})}{\left[\snr{\tilde{\xi}+s\bar{\xi}_{\bar{y}}}^{p}+\tilde{a}(\bar{x})\snr{\tilde{\xi}+s\bar{\xi}_{\bar{y}}}^{q}\right]}-2(2A_{2}+1)\left[\Lambda(n-1)+\lambda\right]+3\lambda\omega_{0}A_{1}.
\end{flalign*}
Since by \eqref{a1},
\begin{flalign}\label{visco4}
-2(2A_{2}+1)\left[\Lambda(n-1)+\lambda\right]+3\lambda\omega_{0}A_{1}>0,
\end{flalign}
we can complete the inequality in the previous display as follows:
\begin{flalign*}
2\ge 2\varepsilon\stackrel{\eqref{sreg}_{2},\eqref{s0}}{\ge}2\nr{\tilde{f}}_{L^{\infty}(B_{1})}\stackrel{\eqref{visco1},\eqref{visco4}}{\ge}-2^{p+1}(2A_{2}+1)\left[\Lambda(n-1)+\lambda\right]+2^{-p}3\lambda\omega_{0}A_{1},
\end{flalign*}
which is not possible, because of \eqref{a1}.\\\\
\emph{Case 2: $\snr{\bar{\xi}}\le s_{0}^{-1}$}. In this case we do not need to rescale \eqref{shift}, but only notice that, by $\eqref{omega}_{1}$ the quantity appearing in \eqref{eneg} can be bounded as
\begin{flalign*}
2(2A_{2}+1)+4A_{1}\omega''(\snr{\bar{x}-\bar{y}})\stackrel{\eqref{rr},\eqref{small12}}{\le} 2(2A_{2}+1)-4A_{1}\beta_{0}(1-\beta_{0})\stackrel{\eqref{a1},\eqref{beta0}}{\le}0,
\end{flalign*}
which means again that at least one eigenvalue of $X_{\delta}-Y_{\delta}$ is negative, thus, by $\eqref{m-}_{2}$, \eqref{alle} and \eqref{eneg} we have
\begin{flalign}
\mathcal{M}^{-}_{\lambda,\Lambda}(X_{\delta}-Y_{\delta})\ge& -2(2A_{2}+1)\left[\Lambda(n-1)+\lambda\right]+4\lambda A_{1}\beta_{0}(1-\beta_{0})\nonumber \\
\stackrel{\eqref{beta0}}{\ge}&-2(2A_{2}+1)\left[\Lambda(n-1)+\lambda\right]+\frac{\lambda A_{1}}{2}.\label{vvisco1}
\end{flalign}
We then compute
\begin{flalign*}
&\min\left\{\snr{\bar{\xi}_{\bar{x}}}^{2},\snr{\bar{\xi}_{\bar{y}}}^{2}\right\}\stackrel{\eqref{small12},\eqref{beta0}}{\ge} \frac{A_{1}^{2}}{16}\snr{\bar{x}-\bar{y}}^{2(\beta_{0}-1)}-2A_{1}A_{2}\snr{\bar{x}-\bar{y}}^{\beta_{0}-1}\nonumber \\
&\qquad\stackrel{\eqref{beta0}}{\ge}A_{1}\snr{\bar{x}-\bar{y}}^{-\frac{1}{2}}\left[\frac{A_{1}}{16}\snr{\bar{x}-\bar{y}}^{-\frac{1}{2}}-2A_{2}\right]\nonumber \\
&\qquad \ \ \ge A_{1}\snr{\bar{x}-\bar{y}}^{-\frac{1}{2}}\left[\frac{A_{1}}{64}\snr{\bar{x}-\bar{y}}^{-\frac{1}{2}}+\frac{3A_{1}}{64}-2A_{2}\right]\stackrel{\eqref{a1},\eqref{rr}}{\ge}\frac{A_{1}^{2}}{64}\rr_{*}^{-1}.
\end{flalign*}
The content of the previous display clearly shows that
\begin{flalign}\label{small1}
\min\left\{\snr{\bar{\xi}_{\bar{x}}},\snr{\bar{\xi}_{\bar{y}}}\right\}\ge \frac{A_{1}}{8}\rr_{*}^{-\frac{1}{2}},
\end{flalign}
thus
\begin{flalign}
\min\{\snr{\bar{\xi}+\bar{\xi}_{\bar{x}}},\snr{\bar{\xi}+\bar{\xi}_{\bar{y}}}\}\ge& \min\left\{\snr{\bar{\xi}_{\bar{x}}},\snr{\bar{\xi}_{\bar{y}}}\right\}-\snr{\bar{\xi}}\ge \min\left\{\snr{\bar{\xi}_{\bar{x}}},\snr{\bar{\xi}_{\bar{y}}}\right\}-s_{0}^{-1}\stackrel{\eqref{a1}}{\ge} 2. \label{small2}
\end{flalign}
At this stage we can derive the viscosity inequalities
\begin{flalign*}
\begin{cases}
\ \left[\snr{\bar{\xi}+\bar{\xi}_{\bar{x}}}^{p}+\bar{a}(\bar{x})\snr{\bar{\xi}+\bar{\xi}_{\bar{x}}}^{q}\right]\bar{F}(X_{\delta})\le \bar{f}(\bar{x})\\
\ \left[\snr{\bar{\xi}+\bar{\xi}_{\bar{y}}}^{p}+\bar{a}(\bar{y})\snr{\bar{\xi}+\bar{\xi}_{\bar{y}}}^{q}\right]\bar{F}(Y_{\delta})\ge \bar{f}(\bar{y})
\end{cases}
\end{flalign*}
and proceed as in \emph{Case 1} to get, with the help of \eqref{vvisco1} and \eqref{visco2},
\begin{flalign*}
\frac{\bar{f}(\bar{x})}{\left[\snr{\bar{\xi}+\bar{\xi}_{\bar{x}}}^{p}+\bar{a}(\bar{x})\snr{\bar{\xi}+\bar{\xi}_{\bar{x}}}^{q}\right]}\ge \frac{\bar{f}(\bar{y})}{\left[\snr{\bar{\xi}+\bar{\xi}_{\bar{y}}}^{p}+\bar{a}(\bar{y})\snr{\bar{\xi}+\bar{\xi}_{\bar{y}}}^{q}\right]}-2(2A_{2}+1)\left[\Lambda(n-1)+\lambda\right]+\frac{\lambda A_{1}}{2}.
\end{flalign*}
Using \eqref{a1} we see that
\begin{flalign*}
-2(2A_{2}+1)\left[\Lambda(n-1)+\lambda\right]+\frac{\lambda A_{1}}{2}>0,
\end{flalign*}
therefore we can use \eqref{small2} to get the following contradiction to \eqref{a1}:
\begin{flalign*}
1\ge 2^{p-1}\left[-2(2A_{2}+1)\left[\Lambda(n-1)+\lambda\right]+\frac{\lambda A_{1}}{2}\right].
\end{flalign*}
Combining the two previous cases, we obtain that if $\bar{u}\in C(B_{1})$ is a normalized viscosity solution of \eqref{20} and $\tilde{x}\in B_{1/2}$, then $\mathcal{L}(\tilde{x})\le 0$ for all $\tilde{x}\in B_{\rr/2}$, which means that, if $\snr{\bar{\xi}}>s_{0}^{-1}$, $\bar{u}$ is Lipschitz-continuous with
\begin{flalign*}
[\bar{u}]_{0,1;B_{\rr/2}}\le c(n,\lambda,\Lambda,p,\rr)
\end{flalign*}
or, if $\snr{\bar{\xi}}\le s_{0}^{-1}$, it is $\beta_{0}$-H\"older continuous with
\begin{flalign*}
[\bar{u}]_{0,\beta_{0};B_{\rr/2}}\le c(n,\lambda,\Lambda,p,\rr)\quad \mbox{for some} \ \ \beta_{0}\in (0,1). 
\end{flalign*}
\end{proof}
\end{proposition}
Notice that the exponent $\beta_{0}$ in the proof of Proposition \ref{ex} does not depend on $\rr$, therefore, if $\bar{u}\in C(B_{1})$ is a normalized viscosity solution of \eqref{20}, regardless to the magnitude of $\snr{\bar{\xi}}$, we can use a standard covering argument to deduce that $\bar{u}\in C^{0,\beta_{0}}_{loc}(B_{1})$ and for any $B_{\rr}\subset B_{1}$ there holds that
\begin{flalign*}
[\bar{u}]_{0,\beta_{0};B_{\rr}}\le c(n,\lambda,\Lambda,p,\rr).
\end{flalign*}
The proof of Theorem \ref{mor} is now complete.
\begin{remark}
\emph{The definitions of $A_{2}$ and $A_{1}$ fix the dependency: $s_{0}=s_{0}(n,\lambda,\Lambda,p,\rr)$. Having a proper look to \eqref{a1} we notice also the presence of $\omega_{0}$, but this really does not matter, since we can set it equal to $1/6$ and we are out of troubles}.
\end{remark}
\section{$C^{1,\gamma}$-local regularity}\label{locreg} We open this section with a $\bar{F}$-harmonic approximation result, which essentially states that, under suitable smallness assumptions, a normalized viscosity solution of problem \eqref{20} in $B_{1}$ can be approximated by a linear function on a smaller ball up to an error which can be controlled via the radius of the ball. 
\begin{lemma}\label{har}
Under assumptions \eqref{ell} and \eqref{assf}-\eqref{assff}, let $\sigma>0$ be as in \eqref{delta} and $\bar{u}\in C(B_{1})$ be a viscosity solution of
\begin{flalign*}
\left[\snr{\bar{\xi}+D\bar{u}}^{p}+\bar{a}(x)\snr{\bar{\xi}+D\bar{u}}^{q}\right]\bar{F}(D^{2}\bar{u})=\bar{f}(x) \ \ \mbox{in} \ \ B_{1},
\end{flalign*}
with $\bar{\xi} \in \mathbb{R}^{n}$ arbitrary. There exists a positive $\iota=\iota(\sigma,n,\lambda,\Lambda,p,q)$ such that if
\begin{flalign*}
\nr{\bar{f}}_{L^{\infty}(B_{1})}\le\iota,
\end{flalign*}
then we can find $\tilde{\xi}_{\sigma}\in \mathbb{R}^{n}$ such that
\begin{flalign*}
\osc_{B_{\sigma}}\left(\bar{u}-\tilde{\xi}_{\sigma}\cdot x\right)\le \frac{\sigma}{2}.
\end{flalign*}
\end{lemma}
\begin{proof}
By contradiction we can find sequences of fully nonlinear operators
\begin{flalign}
\bar{F}_{j}\in C(\mathcal{S}(n),\mathbb{R}) \ \ \mbox{uniformly} \ \ (\lambda,\Lambda)\mbox{-elliptic};\label{01}
\end{flalign}
of vectors $\{\xi_{j}\}_{j \in \N}\subset \mathbb{R}^{n}$ and of functions 
\begin{flalign}\label{02}
&\{\bar{a}_{j}\}_{j \in \N}\subset C(B_{1}) \ \ \mbox{such that} \ \ \bar{a}_{j}(x)\ge 0 \ \ \mbox{for all} \ \ x\in B_{1},\\
&\{\bar{f}_{j}\}_{j \in \N}\in C(B_{1}) \ \ \mbox{with} \ \ \nr{\bar{f}}_{L^{\infty}(B_{1})}\le j^{-1},\label{03}\\
&\{\bar{u}_{j}\}_{j \in \N}\subset C(B_{1})\ \ \mbox{so that} \ \ \sup_{j \in \N}\osc_{B_{1}}(\bar{u}_{j})\le 1 \ \ \mbox{and} \ \ \bar{u}_{j}(0)=0.\label{04}
\end{flalign}
Moreover, $\bar{u}_{j}$ solves
\begin{flalign}\label{05}
\left[\snr{\xi_{j}+D\bar{u}_{j}}^{p}+\bar{a}_{j}(x)\snr{\xi_{j}+D\bar{u}_{j}}^{q}\right]\bar{F}_{j}(D^{2}\bar{u}_{j})=\bar{f}_{j}(x) \ \ \mbox{in} \ \ B_{1},
\end{flalign}
but
\begin{flalign}\label{06}
\osc_{B_{\sigma}}\left(\bar{u}_{j}-\tilde{\xi}\cdot x\right)\ge \frac{\sigma}{2} \quad \mbox{for all} \ \tilde{\xi}\in \mathbb{R}^{n}.
\end{flalign}
The uniformity prescribed by \eqref{01} with respect to assumption \eqref{ell} assures that
\begin{flalign}
\ \{\bar{F}_{j}\}_{j \in \N} \ \ \mbox{converges to a} \ \ (\lambda,\Lambda)\mbox{-elliptic operator} \ \ \bar{F}_{*}\in C(\mathcal{S}(n),\mathbb{R})\label{012}
\end{flalign}
and, by Theorem \ref{mor}, $\bar{u}_{j}\in C^{0,\beta_{0}}_{loc}(B_{1})\cap C(B_{1})$ for some $\beta_{0} \in (0,1)$ thus, using \eqref{mor1} and Arzela-Ascoli theorem we have that
\begin{flalign}\label{07}
\bar{u}_{j}\to \bar{u}_{*} \ \ \mbox{locally uniformly in} \ \ B_{1}. 
\end{flalign}
In particular, from $\eqref{04}_{2}$ and \eqref{07} there holds that
\begin{flalign}\label{08}
\bar{u}_{*}\in C(B_{1}) \ \ \mbox{and} \ \ \osc_{B_{1}}{\bar{u}_{*}}\le 1,
\end{flalign}
but
\begin{flalign}\label{cont}
\osc_{B_{\sigma}}(\bar{u}_{*}-\tilde{\xi}\cdot x)\ge \frac{\sigma}{2}\quad \mbox{for all} \ \ \tilde{\xi}\in \mathbb{R}^{n}.
\end{flalign}
Let us show that $\bar{u}_{*}$ is a viscosity solution of equation
\begin{flalign}\label{homo}
\bar{F}(D^{2}\bar{u}_{*})=0\quad \mbox{in} \ \ B_{1}.
\end{flalign}
To do so, we show that $\bar{u}_{*}$ is a supersolution of \eqref{homo}, then, in a specular way, it can be shown that it is also a subsolution to the same equation, thus concluding the proof. Let $\varphi$ be any test function touching $\bar{u}_{*}$ strictly from below in $\tilde{x}\in B_{1}$. For simplicity, we take $\tilde{x}=0$ thus, by $\eqref{04}_{3}$, $\varphi(0)=\bar{u}_{*}(0)=0$ (recall the comment made at the end of Section \ref{small}) and that $\varphi(x)<\bar{u}_{*}(x)$ for all $x\in B_{\rr}\setminus \{0\}$ for $\rr>0$ sufficiently small. There is no loss of generality in assuming that $\varphi$ is a quadratic polynomial, i.e., 
$$
\varphi(x):=\frac{1}{2}Mx\cdot x+b\cdot x.
$$
In view of \eqref{07}, we see that the polynomial
\begin{flalign*}
\varphi_{j}(x):=\frac{1}{2}M(x-x_{j})\cdot(x-x_{j})+b\cdot(x-x_{j})+\bar{u}_{j}(x_{j})
\end{flalign*}
touches $\bar{u}_{j}$ from below in $x_{j}$ belonging to a small neighborhood of zero. By \eqref{05} we immediately deduce that
\begin{flalign}\label{har2}
\left[\snr{\xi_{j}+b}^{p}+\bar{a}_{j}(x_{j})\snr{\xi_{j}+b}^{q}\right]\bar{F}_{j}(M)\ge \bar{f}_{j}(x_{j}).
\end{flalign}
For the sake of simplicity, from now on we shall distinguish various cases which will eventually lead to the final contradiction.\\\\
\emph{Case 1: sequence $\{\xi_{j}\}_{j\in \mathbb{N}}\subset \mathbb{R}^{n}$ is unbounded.} If the sequence $\{\xi_{j}\}_{j\in \mathbb{N}}$ is unbounded, then we can find a (non relabelled) subsequence such that $\snr{\xi_{j}}\to_{j\to \infty} \infty$ and, choosing $j\in \N$ so large that $\snr{\xi_{j}}\ge \max\left\{1,2\snr{b}\right\}$, by triangular inequality we also have that
\begin{flalign}\label{ass}
\snr{\xi_{j}+b}\ge \snr{\xi_{j}}-\snr{b}\ge \frac{1}{2}\snr{\xi_{j}}\to_{j\to \infty}0,
\end{flalign}
therefore we bound
\begin{flalign}
\left| \ \frac{\bar{f}_{j}(x_{j})}{\left[\snr{\xi_{j}+b}^{p}+\bar{a}_{j}(x_{j})\snr{\xi_{j}+b}^{q}\right]} \ \right |\stackrel{\eqref{03}}{\le}\frac{2^{p}}{j\snr{\xi_{j}}^{p}}\to_{j\to \infty}0.\label{har1}
\end{flalign}
Merging \eqref{har2} and \eqref{har1} we get
\begin{flalign*}
\bar{F}_{*}(M)=\lim_{j\to \infty}\bar{F}_{j}(M)\ge- \lim_{j\to \infty}\left| \ \frac{\bar{f}_{j}(x_{j})}{\left[\snr{\xi_{j}+b}^{p}+\bar{a}_{j}(x_{j})\snr{\xi_{j}+b}^{q}\right]} \ \right |\ge -\lim_{j\to \infty}\frac{2^{p}}{j\snr{\xi_{j}}^{p}}=0,
\end{flalign*}
thus $\bar{F}_{*}(M)\ge 0$.\\\\
\emph{Case 2: sequence $\{\xi_{j}\}_{j\in \N}\subset \mathbb{R}^{n}$ is bounded.} Now we look at the case in which $\{\xi_{j}\}_{j\in \N}$ is bounded. Thus we can extract a (non relabelled) subsequence $\xi_{j}\to_{j\to \infty}\xi_{*}$. As a consequence, $(\xi_{j}+b)\to_{j\to \infty}\xi_{*}+b$. If $\snr{\xi_{*}+b}>0$, then we can find a $\bar{j}\in \N$ so large that 
\begin{flalign*}
\snr{\xi_{j}+b}\ge \frac{1}{2}\snr{\xi_{*}+b}>0\quad \mbox{for all} \ \ j>\bar{j},
\end{flalign*}
thus
\begin{flalign*}
\bar{F}_{*}(M)=\lim_{j\to \infty}\bar{F}_{j}(M)\ge -\lim_{j\to \infty}\left| \ \frac{\bar{f}_{j}(x_{j})}{\left[\snr{\xi_{j}+b}^{p}+\bar{a}(x_{j})\snr{\xi_{j}+b}^{q}\right]} \ \right|\ge -\lim_{j\to \infty}\frac{2^{p}}{j\snr{\xi_{*}+b}^{p}}=0.
\end{flalign*}
Finally, we look at the case $\snr{\xi_{*}+b}=0$. By contradiction, let us assume that
\begin{flalign}\label{contra1}
\bar{F}_{*}(M)<0.
\end{flalign}
By ellipticity, this means that $M$ has at least one positive eigenvalue. Let $\Sigma$ be the direct sum of all the eigensubspaces corresponding to non-negative eigenvalues of $M$ and $\Pi_{\Sigma}$ be the orthogonal projection over $\Sigma$. Since $\snr{\xi_{*}+b}=0$, two situations can occur: $\xi_{*}=-b$ with $\snr{\xi_{*}},\snr{b}>0$ or $\snr{\xi_{*}}=\snr{b}=0$.\\\\
\emph{Case 2.1: $\xi_{*}=-b$ with $\snr{\xi_{*}},\snr{b}>0$.} Since $\varphi$ touches $\bar{u}_{*}$ in zero from below, then, for $\kappa>0$ sufficiently small, by \eqref{07} the map
\begin{flalign}\label{pi2}
\phi_{\kappa}(x):=\frac{1}{2}Mx\cdot x+b\cdot x+\kappa \snr{\Pi_{\Sigma}(x)}
\end{flalign}
touches $\bar{u}_{j}$ in a point $\tilde{x}_{j}$ belonging to a neighborhood of zero. Given that $\sup_{j\in \N}\snr{\tilde{x}_{j}}\le 1$, up to (non relabelled) subsequences, we can assume that $\tilde{x}_{j}\to_{j\to \infty}x_{*}$ for some $x_{*}\in B_{1}$. At this point we examine two scenarios: $\Pi_{\Sigma}(\tilde{x}_{j})=0$ and $\Pi_{\Sigma}(\tilde{x}_{j})\not =0$. If $\Pi_{\Sigma}(\tilde{x}_{j})=0$, then
\begin{flalign}\label{pi1}
\snr{\Pi_{\Sigma}(\tilde{x}_{j})}=\max_{e\in \mathbb{S}^{n-1}}e\cdot \Pi_{\Sigma}(\tilde{x}_{j})=\min_{e\in \mathbb{S}^{n-1}}e\cdot \Pi_{\Sigma}(\tilde{x}_{j}),
\end{flalign}
so, in the light of \eqref{pi1}, the map $\phi_{\kappa}$ can be rewritten as
\begin{flalign}\label{pi3}
\phi_{\kappa}(x)=\frac{1}{2}Mx\cdot x+b\cdot x+\kappa e\cdot \Pi_{\Sigma}(x)\quad \mbox{for all} \ \ e\in \mathbb{S}^{n-1}. 
\end{flalign}
A straightforward computation shows that
\begin{flalign*}
D(e\cdot \Pi_{\Sigma}(x))=\Pi_{\Sigma}(e)\quad \mbox{and}\quad D^{2}(e\cdot \Pi_{\Sigma}(x))=0.
\end{flalign*}
Notice that 
\begin{flalign}\label{pi4}
e\in \mathbb{S}^{n-1}\cap \Sigma\Rightarrow \Pi_{\Sigma}(e)=e\quad \mbox{and}\quad e\in \mathbb{S}^{n-1}\cap \Sigma^{\perp}\Rightarrow \Pi_{\Sigma}(e)=0,
\end{flalign}
where we denoted with $\Sigma^{\perp}$ the orthogonal subspace to $\Sigma$, so if $\snr{Mx_{*}}=0$, we can fix $\bar{j}\in \N$ so large that $\snr{M\tilde{x}_{j}}+\snr{b+\xi_{j}}\le \frac{\kappa}{2}$ for all $j>\bar{j}$, thus
\begin{flalign*}
\snr{M\tilde{x}_{j}+b+\xi_{j}+\kappa e}\ge \frac{\kappa}{2}.
\end{flalign*}
It is easy to see that
\begin{flalign*}
\frac{1}{\left[\snr{M\tilde{x}_{j}+b+\xi_{j}+\kappa e}^{p}+\bar{a}_{j}(\tilde{x}_{j})\snr{M\tilde{x}_{j}+b+\xi_{j}+\kappa e}^{q}\right]}\le \frac{1}{\snr{M\tilde{x}_{j}+b+\xi_{j}+\kappa e}^{p}}\le \frac{2^{p}}{\kappa^{p}},
\end{flalign*}
thus, looking at the viscosity inequality and exploiting the content of the previous display, we have
\begin{flalign}\label{har7}
\bar{F}_{*}(M)=&\lim_{j\to \infty}\bar{F}_{j}(M)\ge\frac{\bar{f}_{j}(\tilde{x}_{j})}{\left[\snr{M\tilde{x}_{j}+b+\xi_{j}+\kappa e}^{p}+\bar{a}_{j}(\tilde{x}_{j})\snr{M\tilde{x}_{j}+b+\xi_{j}+\kappa e}^{q}\right]}\nonumber \\
\ge&-\lim_{j\to \infty}\left| \ \frac{\bar{f}_{j}(\tilde{x}_{j})}{\left[\snr{M\tilde{x}_{j}+b+\xi_{j}+\kappa e}^{p}+\bar{a}_{j}(\tilde{x}_{j})\snr{M\tilde{x}_{j}+b+\xi_{j}+\kappa e}^{q}\right]}\ \right|\nonumber \\
\ge &-\lim_{j\to \infty}\frac{2^{p}}{j\kappa^{p}}=0,
\end{flalign}
so $\tilde{F}_{*}(M)\ge 0$, which contradicts \eqref{contra1}. On the other hand, if $\snr{Mx_{*}}>0$, we first consider the case in which $\Sigma\equiv \mathbb{R}^{n}$ and select $e\in \mathbb{S}^{n-1}$ so that 
\begin{flalign*}
\snr{Mx_{*}+\kappa\Pi_{\Sigma}(e)}\stackrel{\eqref{pi4}_{1}}{=}\snr{Mx_{*}+\kappa e}>0.
\end{flalign*}
We fix $\bar{j}\in \mathbb{N}$ sufficiently large that for $j>\bar{j}$ there holds that 
\begin{flalign}\label{pi5}
\snr{M\tilde{x}_{j}+\kappa e}\ge \frac{1}{2}\snr{Mx_{*}+\kappa e}>0\quad \mbox{and}\quad \snr{\xi_{j}+b}\le \frac{1}{4}\snr{Mx_{*}+\kappa e}.
\end{flalign}
Furthermore, $\Sigma \not \equiv \mathbb{R}^{n}$ means that there exists $e\in \mathbb{S}^{n-1}\cap \Sigma^{\perp}$, thus 
\begin{flalign*}
\snr{M\tilde{x}_{*}+\kappa \Pi_{\Sigma}(e)}\stackrel{\eqref{pi4}_{2}}{=}\snr{Mx_{*}}>0,
\end{flalign*}
so we can fix $\bar{j}\in \mathbb{N}$ so large that
\begin{flalign}\label{pi6}
\snr{M\tilde{x}_{j}}\ge\frac{1}{2}\snr{Mx_{*}}\quad \mbox{and}\quad \snr{b+\xi_{j}}\le \frac{1}{4}\snr{Mx_{*}}.
\end{flalign}
Using either \eqref{pi5} or \eqref{pi6}, we get that
\begin{flalign*}
\snr{M\tilde{x}_{j}+b+\xi_{j}+\kappa\Pi_{\Sigma}(e)}\ge \frac{1}{4}\snr{Mx_{*}+\kappa\Pi_{\Sigma}(e)}>0.
\end{flalign*}
We then compute
\begin{flalign*}
&\frac{1}{\left[\snr{M\tilde{x}_{j}+b+\xi_{j}+\kappa\Pi_{\Sigma}(e)}^{p}+\bar{a}_{j}(\tilde{x}_{j})\snr{M\tilde{x}_{j}+b+\xi_{j}+\kappa \Pi_{\Sigma}(e)}^{q}\right]}\le \frac{2^{2p}}{\snr{Mx_{*}+\kappa \Pi_{\Sigma}(e)}^{p}},
\end{flalign*}
so
\begin{flalign}\label{pi7}
\bar{F}_{*}(M)=&\lim_{j\to \infty}\bar{F}_{j}(M)\ge \lim_{j\to \infty}\frac{\bar{f}_{j}(\tilde{x}_{j})}{\left[\snr{M\tilde{x}_{j}+b+\xi_{j}+\kappa\Pi_{\Sigma}( e)}^{p}+\bar{a}_{j}(\tilde{x}_{j})\snr{M\tilde{x}_{j}+b+\xi_{j}+\kappa \Pi_{\Sigma}(e)}^{q}\right]}\nonumber \\
\ge &-\lim_{j\to \infty}\left| \ \frac{\bar{f}_{j}(\tilde{x}_{j})}{\left[\snr{M\tilde{x}_{j}+b+\xi_{j}+\kappa\Pi_{\Sigma}(e)}^{p}+\bar{a}_{j}(\tilde{x}_{j})\snr{M\tilde{x}_{j}+b+\xi_{j}+\kappa\Pi_{\Sigma}(e)}^{q}\right]} \ \right |\nonumber \\
\ge &-\lim_{j\to \infty}\frac{2^{2p}}{j\snr{Mx_{*}+\kappa\Pi_{\Sigma}(e)}}=0,
\end{flalign}
and we reach again a contradiction to \eqref{contra1}. Now let us consider the occurrence $\snr{\Pi_{\Sigma}(x_{*})}>0$. Choosing $j\in \N$ sufficiently large, we see that $\snr{\Pi_{\Sigma}(\tilde{x}_{j})}>0$ as well, thus the map $x\mapsto \snr{\Pi_{\Sigma}(x)}$ is smooth and convex in a neighborhood of $\tilde{x}_{j}$. Being $\Pi_{\Sigma}$ a projection map, there holds
\begin{flalign}\label{har8}
\Pi_{\Sigma}(x)D(\Pi_{\Sigma}(x))=\Pi_{\Sigma}(x)\quad \mbox{and}\quad D^{2}(\snr{\Pi_{\Sigma}(x)})\ \ \mbox{is non-negative definite},
\end{flalign}
so the variational inequality for $\phi_{\kappa}$ in its original form \eqref{pi2} reads as
\begin{flalign*}
\bar{F}_{j}(M+(D^{2}\snr{\Pi_{\Sigma}})(\tilde{x}_{j}))\ge \frac{\bar{f}_{j}(\tilde{x}_{j})}{\left[\left |M\tilde{x}_{j}+b+\xi_{j}+\kappa\frac{\Pi_{\Sigma}(\tilde{x}_{j})}{\snr{\Pi_{\Sigma}(\tilde{x}_{j})}}\right |^{p}+\bar{a}_{j}(\tilde{x}_{j})\left |M\tilde{x}_{j}+b+\xi_{j}+\kappa\frac{\Pi_{\Sigma}(\tilde{x}_{j})}{\snr{\Pi_{\Sigma}(\tilde{x}_{j})}}\right|^{q}\right]}.
\end{flalign*}
We repeat the same procedure outlined before with $e\equiv e_{j}:=\frac{\Pi_{\Sigma}(\tilde{x}_{j})}{\snr{\Pi_{\Sigma}(\tilde{x}_{j})}}$, thus getting when $\snr{Mx_{*}}=0$,
\begin{flalign}\label{har10}
\bar{F}_{j}(M+(D^{2}\snr{\Pi_{\Sigma}})(\tilde{x}_{j})))\ge -\frac{2^{p}}{j\kappa^{p}}
\end{flalign}
and, for $\snr{M_{x_{*}}}>0$,
\begin{flalign}\label{har11}
\bar{F}_{j}(M+(D^{2}\snr{\Pi_{\Sigma}})(\tilde{x}_{j}))\ge -\frac{2^{2p}}{j\snr{Mx_{*}+\kappa \Pi_{\Sigma}(e)}^{p}}.
\end{flalign}
Passing to the limit in \eqref{har10}-\eqref{har11}, we can conclude that 
\begin{flalign*}
\bar{F}_{*}(M+(D^{2}\snr{\Pi_{\Sigma}})(x_{*}))\ge 0,
\end{flalign*}
therefore, by $\eqref{har8}_{2}$, \eqref{012} and \eqref{ell} we get
\begin{flalign*}
\bar{F}_{*}(M)\ge \bar{F}_{*}(M+(D^{2}\snr{\Pi_{\Sigma}})(x_{*}))\ge 0,
\end{flalign*}
thus contradicting \eqref{contra1}.\\\\
\emph{Case 2.2: $\snr{\xi_{*}}=\snr{b}=0$}. The procedure we are going to follow here is analogous to the one employed for \emph{Case 2.1}, so we will just sketch it. Since $\varphi$ touches $\bar{u}_{*}$ from below in zero, the map
\begin{flalign*}
\phi_{\kappa}(x):=\frac{1}{2}Mx\cdot x+\kappa\snr{\Pi_{\Sigma}(x)}
\end{flalign*}
touches $\bar{u}_{j}$ from below in a point $\tilde{x}_{j}$ belonging to a small neighborhood of zero. Owing to the uniform boundedness of the moduli of the $\tilde{x}_{j}$'s, up to choosing a non-relabelled subsequence, there holds that $\tilde{x}_{j}\to_{j\to \infty}x_{*}$. If $\snr{\Pi_{\Sigma}(x_{*})}=0$, then we see that $\phi_{\kappa}$ can be rewritten as in \eqref{pi3} (with $b\equiv 0$) and it touches $\bar{u}_{j}$ from below in $\tilde{x}_{j}$ for any choice of $e\in \mathbb{S}^{n-1}$. If $\snr{Mx_{*}}=0$, we can fix $\bar{j}\in \N$ so large that $\snr{M\tilde{x}_{j}}+\snr{\xi_{j}}<\frac{\kappa}{2}$, thus, taking any $e\in \mathbb{S}^{n-1}\cap \Sigma$ we have that $\snr{M\tilde{x}_{j}+\xi_{j}+\kappa e}\ge \frac{\kappa}{2}$ and \eqref{har7} follows together with the contradiction to \eqref{contra1}. On the other hand, for $\snr{Mx_{*}}>0$ we can fix $\bar{j}\in \mathbb{N}$ large enough so either \eqref{pi5} or \eqref{pi6} is satisfied with $b=0$ (depending on whether $\Sigma\equiv \mathbb{R}^{n}$ or $\Sigma^{\perp}\not =\{0\}$), so we can safely recover \eqref{pi7} and a contradiction to \eqref{contra1}. Finally, if $\snr{\Pi_{\Sigma}(x_{*})}>0$, we apply our previous construction with $b=0$ to obtain \eqref{har10}-\eqref{har11}, pass to the limit as $j\to \infty$ and finally contradict \eqref{contra1}. \\\\
Merging \emph{Case 1} and \emph{Case 2} we can conclude that $\bar{F}_{*}(M)\ge 0$, so $\bar{u}_{*}$ is a supersolution of \eqref{homo} in $B_{1}$. For the case of subsolutions, we only need to point out that showing that $\bar{u}_{*}$ is a subsolution of equation \eqref{homo} is equivalent to prove that $\tilde{u}_{*}:=-\bar{u}_{*}$ is a supersolution of equation
$$
\tilde{F}_{*}(D^{2}w)=0\quad \mbox{in} \ \ B_{1},
$$
where we set $\tilde{F}_{*}(M):=-\bar{F}_{*}(-M)$, $M\in \mathcal{S}(n)$, which is elliptic in the sense of \eqref{ell}. Hence, we can apply all the previous machinery on $\tilde{u}_{*}$ and conclude that $\bar{u}_{*}$ is a viscosity solution of \eqref{homo}. Proposition \ref{rhar} then applies and $\bar{u}_{*}\in C^{1,\alpha}(B_{1/2})$. In particular, \eqref{0331}, which contradicts \eqref{cont} is in force and the proof is complete.
\end{proof}
\begin{remark}\label{remd}
\emph{In the statement of Lemma \ref{har}, $\iota=\iota(\sigma, n,\lambda,\Lambda,p,q)$, but since \eqref{delta} prescribes that $\sigma=\sigma(n,\lambda,\Lambda)$, we can simply say that $\iota=\iota(n,\lambda,\Lambda,p,q)$.}
\end{remark}
\subsection{Proof of Theorem \ref{reg}} Lemma \ref{har} builds a tangential path connecting normalized viscosity solutions to problem \eqref{20} to normalized viscosity solutions of the limiting profile, for which the Krylov-Safonov regularity theory is available. The core of the proof of Theorem \ref{reg} will be transferring such regularity to normalized viscosity solutions of \eqref{20}. This is the content of the next lemma. 
\begin{lemma}\label{iter} There are $\sigma=\sigma(n,\lambda,\Lambda)\in (0,1)$ and $\gamma=\gamma(n,\lambda,\Lambda,p) \in (0,1)$ such that if $\bar{u}\in C(B_{1})$ is a normalized viscosity solution of \eqref{eqs}, then for any $\kappa \in \N$ there exists $\bar{\xi}_{\kappa}\in \mathbb{R}^{n}$ such that
\begin{flalign}\label{osci}
\osc_{B_{\sigma_{\kappa}}}(\bar{u}-\bar{\xi}_{\kappa}\cdot x)\le \sigma^{\kappa(1+\gamma)}.
\end{flalign}
\end{lemma}
\begin{proof}
Let $\sigma>0$ be as in \eqref{delta} and
\begin{flalign}\label{gamma}
\gamma\in \left(0,\min\left\{\alpha,\frac{1}{p+1},\frac{\log(2)}{-\log(\sigma)}\right\}\right),
\end{flalign}
where $\alpha\in (0,1)$ is the same as in \eqref{030}. A direct consequence of the restriction in \eqref{gamma} is that
\begin{flalign}\label{gamma2}
\sigma^{\gamma}>\frac{1}{2}.
\end{flalign}
Moreover, we fix the parameter $m\in (0,1)$ defined in \eqref{mm} equal to $\iota^{\frac{1}{p+2}}$, where, by Remark \ref{remd}, $\iota=\iota(n,\lambda,\Lambda,p,q)>0$ is the same as in Lemma \ref{har} corresponding to $\sigma$ in \eqref{delta}. In this way we also determine the dependency $m=m(n,\lambda,\Lambda,p,q)$, thus closing the ambiguity due to the presence of $m$ in the scaled problem \eqref{20}. Notice that none of the quantities appearing in the estimates provided so far depend on the sup-norm of $\bar{a}$, nor on its modulus of continuity. For $\kappa \in \N\cup \{0\}$, set $\sigma_{\kappa}:=\sigma^{\kappa}$. We then proceed by induction.\\\\
\emph{Basic step: $\kappa=0$}. In this case, \eqref{osci} is verified with $\bar{\xi}_{0}=0$. In fact,
\begin{flalign*}
\osc_{B_{\sigma_{0}}}(\bar{u}-\bar{\xi}_{0}\cdot x)=\osc_{B_{1}}(\bar{u})\stackrel{\eqref{19}_{1}}{\le}1.
\end{flalign*}
\emph{Induction assumption}. We assume that there exists $\bar{\xi}_{\kappa}\in \mathbb{R}^{n}$ such that
\begin{flalign}\label{osck}
\osc_{B_{\sigma_{\kappa}}}(\bar{u}-\bar{\xi}_{\kappa}\cdot x)\le \sigma_{\kappa}^{1+\gamma}.
\end{flalign}
\emph{Induction step}. Define
\begin{flalign*}
\bar{u}_{\kappa}(x):=\sigma_{\kappa}^{-(1+\gamma)}\left[\bar{u}(\sigma_{\kappa}x)-\sigma_{k}\bar{\xi}_{\kappa}\cdot x\right].
\end{flalign*}
It is easy to see that $\bar{u}_{\kappa}\in C(B_{1})$ satisfies
\begin{flalign*}
\left[\snr{D\bar{u}_{\kappa}+\tilde{\xi}_{\kappa}}^{p}+\bar{a}_{\kappa}(x)\snr{D\bar{u}_{\kappa}+\tilde{\xi}_{\kappa}}^{q}\right]\bar{F}_{\kappa}(D^{2}\bar{u}_{\kappa})=\bar{f}_{\kappa}(x)\quad \mbox{in} \ \ B_{1},
\end{flalign*}
where 
\begin{flalign*}
&\tilde{\xi}_{\kappa}:=\sigma_{\kappa}^{-\gamma}\bar{\xi}_{\kappa},\\
&\bar{a}_{\kappa}(x):=\sigma_{\kappa}^{\gamma(q-p)}\bar{a}(\sigma_{\kappa}x)\quad \mbox{for all} \ \ x\in B_{1}\\
&\bar{F}_{\kappa}(M):=\sigma_{\kappa}^{1-\gamma}\bar{F}(\sigma_{\kappa}^{\gamma-1}M)\quad \mbox{for all} \ \ M\in \mathcal{S}(n)\\
&\bar{f}_{\kappa}(x):=\sigma_{\kappa}^{1-\gamma(1+p)}\bar{f}(\sigma_{\kappa}x).
\end{flalign*}
Notice that, by \eqref{ell}, $\bar{F}_{\kappa}$ is uniformly $(\lambda,\Lambda)$-ellitic with the same ellipticity constants as $\bar{F}$. By the choice we made on $m=\iota^{\frac{1}{p+2}}$, we obtain that
\begin{flalign*}
\nr{\bar{f}_{\kappa}}_{L^{\infty}(B_{1})}\le \sigma_{\kappa}^{1-\gamma(1+p)}\nr{\bar{f}}_{L^{\infty}(B_{1})}\stackrel{\eqref{gamma},\eqref{19}_{3}}{\le}\iota.
\end{flalign*}
Finally, by \eqref{osck} we readily see that
\begin{flalign*}
\osc_{B_{1}}(\bar{u}_{\kappa})=\sigma_{\kappa}^{-(1+\gamma)}\osc_{B_{\sigma_{\kappa}}}\left(\bar{u}-\bar{\xi}_{\kappa}\cdot x\right)\le 1,
\end{flalign*}
therefore all the assumptions of Lemma \ref{har} are satisfied, thus there exists $\tilde{\xi}_{\kappa+1}$ such that
\begin{flalign}\label{200}
\osc_{B_{\sigma}}\left(\bar{u}_{\kappa}-\tilde{\xi}_{\kappa+1}\cdot x\right)\le \frac{\sigma}{2}.
\end{flalign}
Set $\bar{\xi}_{\kappa+1}:=\bar{\xi}_{\kappa}+\sigma_{\kappa}^{\gamma}\tilde{\xi}_{\kappa+1}$. We then have
\begin{flalign*}
\osc_{B_{\sigma_{\kappa+1}}}\left(\bar{u}-\bar{\xi}_{\kappa+1}\cdot x\right)\stackrel{\eqref{200}}{\le} \frac{1}{2}\sigma \cdot \sigma_{\kappa}^{1+\gamma}\stackrel{\eqref{gamma2}}{\le} \sigma_{\kappa+1}^{1+\gamma},
\end{flalign*}
and we are done.
\end{proof}
Once Lemma \ref{iter} is available, we can complete the proof of Theorem \ref{reg} in a straightforward way. Whenever $\rr\in (0,1]$, we can find $\kappa \in \mathbb{N}\cup\{0\}$ such that $\sigma^{\kappa+1}< \rr\le \sigma^{\kappa}$. So we have
\begin{flalign*}
\osc_{B_{\rr}}\left(\bar{u}-\bar{\xi}_{\kappa}\cdot x\right)\le \osc_{B_{\sigma^{\kappa}}}\left(\bar{u}-\bar{\xi}_{\kappa}\cdot x\right)\stackrel{\eqref{osci}}{\le}\sigma^{\kappa(1+\gamma)}\le \sigma^{-(1+\gamma)}\rr^{1+\gamma}:=c(n,\lambda,\Lambda,p)\rr^{1+\gamma},
\end{flalign*}
therefore $\bar{u}$ is $C^{1,\gamma}$ around zero. This is enough, in fact by standard translation arguments we can prove the same for any point of $B_{1/2}$ thus getting that $\bar{u}\in C^{1+\gamma}(B_{1/2})$ and then conclude with Remark \ref{rem1} and a covering argument.

\end{document}